\documentclass[11pt]{amsart}
\usepackage{amssymb,latexsym,amsmath}
\usepackage{bbm}
\usepackage{graphicx}

 \theoremstyle{plain}            
 \newtheorem{theorem}{Theorem}[section]
 \newtheorem*{maintheorem*}{Main Theorem}
 \newtheorem*{maintheorem.}{Main Theorem~1'}
 
 \newtheorem{proposition}[theorem]{Proposition}
 \newtheorem{lemma}[theorem]{Lemma}
 \newtheorem{corollary}[theorem]{Corollary}

 \theoremstyle{definition}       

 \theoremstyle{remark}
 \newtheorem{remark}[theorem]{Remark}

\newcommand{\eps}{\varepsilon} 
\DeclareMathAlphabet{\Mb}{U}{msb}{m}{n}

  \def\C{\Mb{C}}
  \def\N{\Mb{N}}
  
  \def\R{\Mb{R}}
  \def\Z{\Mb{Z}}

\newcommand{\vth}{\vartheta}

\def\cc{{\mathcal C}}
\def\cd{{\mathcal D}}

\def\ch{{\mathcal H}}

\def\cl{{\mathcal L}}
\def\cam{{\mathcal M}}
\def\cn{{\mathcal N}}
\def\co{{\mathcal O}}
\def\cp{{\mathcal P}}

\def\car{{\mathcal R}}
\def\cs{{\mathcal S}}

\def\cw{{\mathcal W}}

\DeclareMathOperator{\sym}{Sym}
\DeclareMathOperator{\supp}{supp}
\def\uni{u}

\title[MODULAR THEORY IN LOCAL QUANTUM PHYSICS]{MODULAR THEORY FOR THE VON~NEUMANN ALGEBRAS OF LOCAL QUANTUM PHYSICS}
\author{Daniele Guido}
\address{Department of Mathematics, Universit\`a di Roma "Tor Vergata"}
\email{guido@mat.uniroma2.it}
 \date{\today}

\begin{document}
\begin{abstract}

In the first part, the second quantization procedure and the free Bosonic scalar field will be introduced, and the axioms for quantum fields and nets of observable algebras will be discussed.

The second part is mainly devoted to an illustration of the Bisognano-Wichmann theorem for Wightman fields and in the algebraic setting,  with a discussion on the physical meaning of this result.

In the third part some reconstruction theorems based on modular groups will be described, in particular the possibility of constructing an action of the symmetry group of a given theory via modular groups, and the construction of free field algebras via representations of the symmetry group. 

\end{abstract}
\maketitle
\tableofcontents

\section{Introduction}
Since the beginning of modular theory for von~Neumann algebras in the late sixties \cite{ToTa}, the relations with quantum physics became apparent, first through the interpretation of the analyticity properties of the modular automorphism group as a condition for thermodynamical equilibrium due to Haag, Hugenoltz and Winnink \cite{HHK}, and few years later with the theorem of Bisognano and Wichmann \cite{BW1,BW2}, relating the symmetries of quantum field theory on Minkowski space with the modular objects associated with suitable regions.
The mathematical instruments needed to prove and consolidate these relations are the theory of operator algebras together with its connections with the theory of analytic functions in one or more variables and the theory of group representations, plus some geometrical properties of the spacetimes.
In this notes we present  the descriptions of quantum field theory by the G{\aa}rding-Wightman axioms and by the Haag-Kastler axioms together with the common example of the free scalar field, showing in particular how, for the free Bose fields, the local subspaces and the associated modular objects in the one-particle space are connected with the algebras of local observables and the associated modular objects on the Fock space. We refer to \cite{Lledo} for the analogous connection in the case of Fermi fields (CAR algebras). Then we discuss the relations of the modular objects for the algebras of observables of some regions with the symmetry group of the quantum theory, both in the concrete form of the Bisognano-Wichmann theorem for Wightman fields and in the more abstract form of the Borchers theorem. Finally we analyze some reconstruction results in this context, e.g. reconstruct modular objects via symmetries of a net of von~Neumann algebras, reconstruct the symmetries of a net of von~Neumann algebras via modular objects, reconstruct a net of von~Neumann algebras via a representation of the symmmetry group.

\section{Free quantum fields and local quantum theories}\label{Free&Local}

 \subsection{Free scalar field}
 As a first basic example of a quantum field, I will introduce the free quantum scalar field (of mass $m$) on the Minkowski space. This example is ubiquitous, so I  quote here only some basic references, the book of Streater \& Wightman \cite{SW}, chapter X.7 of the book of Reed and Simon \cite{RS2}, and
the papers of Araki \cite{Araki1,Araki3}, by which this presentation has been strongly influenced. Other references will be given below.

\subsubsection{Spacetime and Symmetries}\label{1pSp}

Let me recall that the (four-dimensional) Minkowski space $M^4$ is the real manifold $\R^4$ with the Minkowski pseudometric given by the signature $(+,-,-,-)$, the first component describing the time coordinate, the others the space coordinates. The group of diffeomorphisms of $M^4$ preserving this (pseudo)-metric is the so called Poincar\'e group $\cp$, which is given by the semi-direct product $\cl \times \R^4$ of the Lorentz group $\cl$ with the translations, where $\cl$ is the group of isometries of $M^4$ as a vector space with the (indefinite) inner product $x\cdot y=x_0y_0-x_1y_1-x_2y_2-x_3y_3$. The Lorentz group is a finite dimensional locally compact Lie group. We denote by $\cl_+$ the proper Lorentz group, namely the subgroup of orientation preserving transformations, and by  $\cl^\uparrow$ the orthochronous Lorentz group, namely the subgroup of  time-orientation preserving transformations.  $\cl_+$ is also denoted by $SO(1,3)$.
The intersection $\cl_+^\uparrow=\cl_+\cap \cl^\uparrow$ is the identity component of $\cl$. Subgroups $\cp_+^\uparrow=\cp_+\cap \cp^\uparrow$ of the Poincar\'e group are defined analogously.

Elements of the Lorentz group can be represented as $4\times4$ matrices. Elements of the form $\begin{pmatrix}1&0\\0&V\end{pmatrix}$, where $V$ is an orthogonal transformation on $\R^3$, represent a change of the space coordinates. Elements of the form $$\begin{pmatrix}\cosh\vth&-\sinh\vth&0&0\\
-\sinh\vth&\cosh\vth&0&0\\
0&0&1&0\\
0&0&0&1
\end{pmatrix}\quad,\qquad\vth\in\R,$$
are called {\it boosts}, and relate two reference frames in relative uniform motion. In the example above, the motion is along the axis $x_1$, with velocity $v$ given by $\cosh\vth=(1+v^2/c^2)^{-1/2}$.

Only {\it positive energy representations} of $\cp_+^\uparrow$ will be considered, namely representations for which the joint spectrum of the unitaries implementing the translation group is contained in the closed forward light cone $\overline{V}_+=\{x\in M^4: x\cdot x\geq0, x_0\geq 0\}$. Such condition is also called {\it spectrum condition}.

Positive energy, irreducible representations of the (proper orthochronous) Poin\-car\'e group are labeled by two parameters, the mass $m\geq 0$ and the spin $s\in\N$ (indeed for $m=0$ also ``infinite spin" is allowed).

More precisely, since $\cp_+^\uparrow=\cl_+^\uparrow \times \R^4$, irreducible positive energy representations are constructed as follows: consider the action of  $\cl_+^\uparrow $ on $ \R^4$ and choose an orbit contained in the closed forward light cone with the origin excluded. These orbits are the mass hyperboloids $H_m=\{p\in M^4:p\cdot p = m^2, p_0>0\}$, $m\geq0$. Then choose an irreducible representation of the stabilizer of a point in the given orbit. For $m>0$, the stabilizer is  (isomorphic to) $SO(3)$, the group of orientation preserving rotations; for $m=0$ the stabilizer is (isomorphic to) $E(2)$, the Euclidean group of the plane. Finally use Mackey induction to get a representation of the Poincar\'e group (cf. e.g. \cite{Lip,Simms}).
Since irreducible representations of $SO(3)$ are parametrized by the spin $s=0,1,2,\dots$, irreducible representations of $\cp_+^\uparrow$ are labeled by the pair $(m,s)$.

The irreducible representations of  $E(2)$ are of two kinds: finite dimensional, when the translational part is trivially represented, and infinite-dimensional, otherwise. The first are just representations of the circle group, hence are labeled by the so-called helicity $s\in\Z$. They give rise to the mass zero, helicity $s$ representations. The latter give rise to the so called {\it infinite spin} representations, cf. section \ref{modloc}.

\subsubsection{One-particle space} Let me describe the mass $m$, spin 0 irreducible representation ${\uni}$ of the (proper, orthochronous) Poincar\'e group. It acts on the Hilbert space $L^2(H_m,\ d\Omega_m)$ of square-summable functions on the mass-hyperboloid $H_m$, w.r.t. the Lorentz invariant measure $d\Omega_m$. It is called the one-particle space for the quantum fields describing particles of mass $m$ and spin $s$. The representation ${\uni}$ acts follows:

\begin{equation}\label{PoiAct}
({\uni}(\Lambda,a)\varphi)(p)=e^{ia\cdot p}\varphi(\Lambda^{-1}p), \qquad \Lambda\in \cl^\uparrow_+ ,  \ a\in R^4,
\end{equation}
where elements of $\cp_+^\uparrow$ are written as pairs $(\Lambda,a)$ with $\Lambda\in \cl^\uparrow_+ ,  \ a\in R^4$. Since the measure $d\Omega_m$ is Lorentz invariant, the action is unitary.

Let us notice that, if we wish to extend the representation to the proper Lorentz group, it is sufficient to describe ${\uni}(\gamma,0)$, where $\gamma$ represents the space-time reflection. If  $P$ denotes the four component generator of the translation subgroup, we have ${\uni}(\gamma,0)e^{ia\cdot P}{\uni}(\gamma,0)=e^{-ia\cdot P}$, which implies ${\uni}(\gamma,0)i P{\uni}(\gamma,0)=-i P$. In order to have positive energy, ${\uni}(\gamma,0)$ has to be conjugate-linear. We call (anti)-unitary a representation of $\cp_+$ such that ${\uni}(g)$ is unitary if $g$ is time-preserving and is anti-unitary if $g$ is time-reversing. In quantum field theory, the anti-unitary implementing the reflection $\gamma$ is called the PCT (parity, charge conjugation, time) transformation, and is denoted by $\Theta$.

We shall consider the following real-linear embedding of the space $\cs(M^4,\R)$ of smooth rapidly decreasing functions into the complex Hilbert space $L^2(H_m,\ d\Omega_m)$:
\begin{equation}\label{localization}
f\in\cs(M^4,\R)\to Ef\in L^2(H_m,\ d\Omega_m), Ef(p)=\sqrt{2\pi}\hat{f}(p), p\in H_m.
\end{equation}
$\hat f$ denoting the Fourier transform (up to a sign). 
Proposition below follows by a direct computation.

\begin{proposition}\label{poincareaction}
Setting $f_{(\Lambda,a)}(x)=f(\Lambda^{-1}(x-a))$, $\Lambda\in \cl^\uparrow_+, a\in\R^4$, we have
\begin{equation}\label{cov1}
Ef_{(\Lambda,a)}={\uni}(\Lambda,a)Ef.
\end{equation}
\end{proposition}

\begin{remark}
Equation (\ref{cov1}) can be used as a prescription for the definition of  ${\uni}(\gamma,0)$. A straightforward computation gives ${\uni}(\gamma,0)\varphi=\overline\varphi$.
\end{remark}

\subsubsection{Local structure of the one-particle space} \label{locstruc}
Now, for any bounded open region $\co\subset M^4$, let us define the corresponding local space as the real-linear closed vector space given by
\begin{equation}\label{localspaces}
K(\co)=\{Ef:f\in\cs(M^4,\R), {\rm supp\ }f\subset\co\}^-.
\end{equation}
By Proposition \ref{poincareaction}, $K(g\co)={\uni}(g)K(\co)$, for any $g\in\cp^\uparrow_+$.

For unbounded regions $\cc$, we set $K(\cc)=\bigvee_{\co\subset\cc}K(\co)$, $\vee$ denoting the generated real-linear closed space.

For any open region $\co\subset M^4$, we consider its (open) space-like complement $\co'$, with
\begin{equation}\label{slc}
\co'=\{x\in M^4: (x-y)^2 < 0, y\in\co\}.
\end{equation}

For any real-linear closed subspace $K\subset  L^2(H_m,\ d\Omega_m)$, we set
\begin{equation}\label{symplectic complement}
K'=\{h\in L^2(H_m,\ d\Omega_m): \Im(h,k)=0, k\in K\}.
\end{equation}

The following theorem has been proven by Araki \cite{Araki3}

\begin{theorem}\label{lattices1}
For any bounded open nonempty simply connected region $\co$ with regular boundary we have
\item{$(i)$}   $K(\co')=K(\co)'.$
\item{$(ii)$}  $K(\co)\cap iK(\co)=\{0\}.$
\item{$(iii)$} $K(\co)+ iK(\co)$ is dense.
\end{theorem}

\subsubsection{Tomita-Takesaki theorem for real subspaces}
A closed real subspace $K$ satisfying properties $(ii)$ and $(iii)$ of Theorem \ref{lattices1} is called {\it standard}. Given a standard subspace $K$, let us consider the following operator:
\begin{equation}\label{1pT}
\begin{matrix}
s_K:&K+iK&\to& K+iK\\
        &h+i k&\mapsto & h-i k
\end{matrix}
\end{equation}

\begin{proposition}\label{1pTTthm} There is a one-to-one correspondence between standard subspaces and closed, anti-linear operators $s$ satisfying $\cd(s)=Rg(s)$ and $s^2=1$.
Let now $K$ be a standard subspace, $s_K$ be as above, $s_K=j_K\delta_K^{1/2}$ be the polar decomposition. We have:
\item{$(i)$}  $s_{K'}=s_K^*$ and $K\cap K'=\{x:j_Kx=x\ \&\ \delta_K x=x\}$.
\item{$(ii)$} $j_K K=K',\quad \delta_K^{it}K=K.$
\end{proposition}

\begin{proof}
Assume $K$ is standard. The operator $s_K$ is clearly well defined, since property  $K\cap iK=\{0\}$ implies that a vector in $K+iK$ can be uniquely decomposed as $k_1+ik_2$, with $k_1,k_2\in K$, and densely defined. A simple computation shows that $\|k_1+ik_2\|^2_{G(s_K)}=2(\|k_1\|^2+\|k_2\|^2)$, namely $(K+iK,\|\cdot\|_{G(s_K)})$ is isomorphic, as a real Hilbert space, to $K\oplus K$, hence is closed. Properties $\cd(s_K)=Rg(s_K)$ and $s_K^2=1$ are now obvious. Conversely, if $s$ has the mentioned properties, any vector $x\in\cd(s)$ can be uniquely decomposed as a sum of an invariant and of an anti-invariant vector for $s$, $x=\frac12(x+sx)+\frac12(x-sx)$, $x$ is invariant {\it iff} $ix$ is anti-invariant, hence setting $K:=\{x:sx=x\}$ we get a closed standard subspace. In the following we drop the subscript $K$ from the operators $s,j,\delta$ when no confusion arises.

$(i)$. From the properties above we get $j^2=1$, $j\delta j=\delta^{-1}$, and $s^*=j\delta^{-1/2}$, hence $\cd(s^*)=Rg(s^*)$ and $(s^*)^2\subset 1$. This implies that,  as for the operator $s$, $s^*$ is determined by its invariant subspace $\{x\in\ch:s^*x=x\}$.

Let us recall that, since $s$ is antilinear, $(sx,y)=(s^*y,x)$.  If $s k=k$ and $s^*h=h$, we have $(k,h)=(sk,h)= (s^*h,k)=(h,k)$, namely $\Im (k,h)=0$. As a consequence, $s_{K'}\supseteq s_K^*$. Conversely, if $k_1, k_2\in K$ and $h_1,h_2\in K'$, one gets, by a straightforward computation, $(h_1 + i h_2 , s(k_1+ i k_2))=(k_1 + i k_2,h_1 - i h_2)$, hence  $h_1 + i h_2\in\cd(s^*_{K})$ and $s^*_{K}(h_1 + i h_2)=h_1 - i h_2$. This prove the equality $s_{K'}=s_K^*$. Then, $x\in K\cap K'$ {\it iff} $s_Kx=x$ and $s_{K}^*x=x$, namely $\delta x=s^*s x=x$, from which $\delta^{1/2} x=x$ and $jx=x$. The converse implication is proved analogously.

$(ii)$. Let me first give the idea of the proof: we may decompose the complex Hilbert space $\ch$ as a direct integral of 2-dimensional spaces (possibly up to the eigenspace $\{x=\delta x\}$, which may be odd-dimensional), in such a way that $K$, $K'$, $s$ and $s^*$ are decomposed accordingly. In any such fiber, the operators $j$ and $\delta$ can be written as
$$
j=\begin{pmatrix}0&C\\C&0\end{pmatrix}\quad
\delta=\begin{pmatrix}\tan^2\frac\vth2&0\\0&\tan^{-2}\frac\vth2\end{pmatrix},
$$
where $C$ denotes the complex conjugation and $\vth\in(0,\pi/2]$, the space $K$ is generated by the vectors 
$$
y_+=\begin{pmatrix}\cos\frac\vth2\\\sin\frac\vth2\end{pmatrix}\quad
y_-=\begin{pmatrix}i\cos\frac\vth2\\-i\sin\frac\vth2\end{pmatrix},
$$
and the space $K'$ is generated by the vectors 
$$
y'_+=\begin{pmatrix}\sin\frac\vth2\\\cos\frac\vth2\end{pmatrix}\quad
y'_-=\begin{pmatrix}i\sin\frac\vth2\\-i\cos\frac\vth2\end{pmatrix}.
$$
Form this one gets $jy_\pm=y'_\pm$, hence $j K=K'$, and $ \delta^{it}y_\pm=\cos[(\log\tan^2\frac\vth2)t] y_\pm \pm \sin[(\log\tan^2\frac\vth2)t] y_\mp$, hence $\delta^{it}K=K$.

The previous argument can be made rigorous as in \cite{GuLo4}, Proposition 1.14, where we are assuming for symplicity that $1\not\in\sigma_p(\delta)$:

Let us choose a selfadjoint antiunitary $C$ commuting with $j$
and $\delta$, and set $U=jC$, so that $U(\log\delta) U=-\log\delta$. Then denote
with $\cl$ the real vector space of $C$-invariant vectors in the
spectral subspace $\{\log\delta>0\}$ and by $\psi^{\pm}$ the maps $\psi^{+}
:y\in \cl\mapsto U\cos\frac\Theta2 y+\sin\frac\Theta2 y$, $\psi^{-}:y\in\cl\mapsto
iU\cos\frac\Theta2 y-i\sin\frac\Theta2 y$, where the operator $\Theta$ is defined by
$|\log\delta|=-2\log\tan\frac\Theta2$, $\sigma(\Theta)\subseteq[0,\pi/2]$.

Since $U$ maps the spectral space $\{\log\delta>0\}$ onto the spectral
space $\{\log\delta<0\}$, both $\psi^{+}$ and $\psi^{-}$ are isometries,
and a simple calculation shows that their ranges are real-orthogonal.
Moreover, decomposing $\ch$ as $\{\log\delta<0\}\oplus\{\log\delta>0\}$, one
can show that any solution of the equation $sx=x$ can be written as a
sum $\psi^{+}(y)+\psi^{-}(z)$, namely the map $\psi^{-}+\psi^{+} :
\cl\oplus_{\R} \cl\to K$ is an isometric isomorphism of real Hilbert
spaces. For a more detailed proof and the relation of $\Theta$ with the angle between $K$ and $iK$, see \cite{FiGu2}.
\end{proof}

\begin{remark}\label{relation}
Let us notice that, if $\car$ is a von~Neumann algebra acting on a Hilbert space $\ch$ with a standard vector $\Omega$, the closure $K$ of the real space $\car_{sa}\Omega$ is standard, and the Tomita operator $S$ coincides with the operator $s_K$ considered above. Compare the statement of the previous theorem with that of the Tomita-Takesaki theorem in this volume \cite{Lledo}
\end{remark}

\subsubsection{Second quantization functor}

 Let $\ch$ be a complex separable Hilbert
space. The {\it symmetric Fock space} over it is
$$e^\ch=\bigoplus^\infty_{n=0}\ch^{\otimes_{\sym} n}$$
where $\ch^{\otimes_{\sym} n}$ is the subspace of the $n-$th tensor product of
$\ch$ which is pointwise invariant under the natural action of the
permutation group. More precisely, $\ch^{\otimes_{\sym}n}=\sym (\ch^{\otimes n})$, where  the orthogonal projection $\sym$ is defined as
$$
\sym (x_1\otimes\dots\otimes x_n):=\frac1{n!}\sum_{\sigma\in P(n)}x_{\sigma(1)}\otimes\dots\otimes x_{\sigma(n)}.
$$
$\ch^{\otimes_{\sym} n}$ is called the $n$-particle space\footnote{The reason for $\sym$ is to the fact that we describe quantum particles (hence indistinguishable particles) obeying Bose-Einstein statistics.}.

The set of {\it coherent} vectors in $e^\ch$ consists of the vectors
$$e^h=\bigoplus^\infty_{n=0}\frac{h^{\otimes n}}{\sqrt{n!}}.$$
This set turns out to be total in $e^\ch$ (see property $(a)$ below).

The first  important class of operators acting on $e^\ch$ is that of {\it second quantization operators}.
For any closed, densely defined, operator $a$ on the one-particle space $\ch$,  we set
$$
e^a=\bigoplus^\infty_{n=0}a^{\otimes n}\ ,
$$
on the linear span of the symmetrized elementary tensors on $\cd(a)$, namely of the vectors $\sym (x_1\otimes\dots\otimes x_n)$, with $x_i\in\cd(a)$, $i=1,\dots, n$, so that $e^a$ is densely defined. Let us observe that $(e^a)^*\supseteq e^{(a^*)}$, and since the latter is densely defined, $e^a$ is indeed closable, cf. \cite{RS1} Theorem VIII.1. In the following, we shall denote its closure with the same symbol $e^a$. 

In the particular case in which  $u$ is unitary on the one-particle space, $e^u$ is unitary on the Fock space. Setting $U(g)=e^{u(g)}$, $g\in\cp_+^\uparrow$, for the representation $u$ of the Poincar\'e group on $\ch$, we get a positive energy representation $U$ on the Fock space.

The second class consists of {\it Weyl unitaries}, which are the range of the map 
$$
h\to W(h)
$$ 
from $\ch$ to the unitaries on $e^\ch$ defined by
\begin{align*}
 W(h)e^0&=\exp\left(-\frac14\|h\|^2\right)
e^{\frac i{\sqrt2}h},\quad h\in\ch\\
W(h)W(k)&=\exp\left(-\frac i2 \Im(h,k)\right)W(h+k)\quad h,k\in\ch
\end{align*}

The vector $e^0=\Omega$ is called {\it vacuum} and the relations in the
last equality  are called Canonical Commutation Relations (CCR). We refer to \cite{Lledo} for the treatment of fields obeying Canonical Anticommutation Relations (CAR). 

Via the preceding equalities $W(h)$ becomes a well defined, isometric and
invertible (with inverse $W(-h)$) operator on the dense set spanned
by coherent vectors, and hence it extends to a unitary on $e^\ch$. Weyl
unitaries generate the so-called {\it second quantization algebras}.
With each closed real linear subspace $K\subset\ch$, a von~Neumann 
algebra $\car(K)$ is associated, defined by 
$$
\car(K)=\{W(h),\quad h\in K\}''.
$$ 

The following theorems give some properties of the second quantization algebras and their modular operators.

\begin{theorem}[\cite{EO}]\label{thm:2qS} A second quantization algebra $\car(K)$ is
in standard form w.r.t. the vacuum if and only if $K$ is standard.
In this case $S=e^s$, $\Delta=e^\delta$, and $J=e^j$, where $S$ is the
Tomita operator of $(\car(K),e^0)$ and $S=J\Delta^{1/2}$, $s=j\delta^{1/2}$ are
the polar decompositions of $S$ and $s$, respectively.
\end{theorem}

\begin{theorem}[\cite{Araki1}] \label{thm:lattices}The map $K\to\car(K)$ is an
isomorphism of complemented nets, where the complementation of an
algebra is its commutant and the complementation of a real subspace
$K$ is the simplectic complement $K'=\{h\in \ch:\Im(h,k)=0\}$.
\end{theorem}

\subsubsection{Some proofs}
We shall now prove the main results concerning second quantization algebras, and in particular the following results from Theorems \ref{thm:2qS} and \ref{thm:lattices}:
\begin{itemize}
\item[$(a)$] The vacuum vector $\Omega$ is cyclic and  separating for the second quantization algebra $\car(K)$  if and only if $K$ is standard. In particular, in this case the set $\{e^k:k\in K\}$ is total in $e^\ch$.
\item[$(b)$] If $K$ is standard, $S=e^s$, $\Delta=e^\delta$, and $J=e^j$. In this case $\car(K)'=\car(K')$. 
 \end{itemize}

\begin{lemma}\label{lemma:coherentvectors}
Let $a$ be a closed, densely defined, operator on the one-particle space $\ch$. Then the vectors $e^h$, with $h\in\cd(a)$, belong to $\cd(e^a)$, and $e^ae^h=e^{ah}$.
\end{lemma}
\begin{proof}
For any $h\in\ch$, 
$$
\|e^h-\bigoplus_{n=0}^N\frac{h^{\otimes n}}{\sqrt{n!}}\|^2\leq\sum_{n=N+1}^\infty\frac{\|h\|^{2n}}{n!}\to0,
\text{for}\ N\to\infty,
$$
as a consequence, when $h\in\cd(a)$, $\oplus_{n=0}^N\frac{h^{\otimes n}}{\sqrt{n!}}$ converges to $e^h$ in the graph norm of $e^a$. The thesis follows.
\end{proof}

\begin{lemma}\label{lemma:sym}
Symmetrized elementary tensors can be written as linear combinations of tensor powers $x^{\otimes n}$, more precisely
\begin{equation}
\sym (x_1\otimes\dots\otimes x_n)=\frac1{n!}
\sum_{F\subseteq(n)}(-1)^{|F|+n}\left(\sum_{j\in F}x_j\right)^{\otimes n},
\end{equation}
where $(n)$ is the set of the first $n$ natural numbers, and $|F|$ denotes the cardinality of the subset $F$. 
\end{lemma}
\begin{proof}
In the following,  $I$ will denote a multi-index in $(\N\cup\{0\})^n$. We set 
$$|I|=\sum_{j=1}^n	I_j,\ \supp(I)=\{j:I_j\ne0\},\ \begin{pmatrix}n\\ I\end{pmatrix}=\frac{n!}{I_1!\cdot\dots\cdot I_n!},\ \vec{x}^{\otimes I}=\bigotimes_{j\in\supp(I)}x_j^{\otimes I_j}.
$$ 
Since under symmetrization the order in the tensor product does not matter, we have
\begin{align*}
\left(\sum_{j\in F}x_j\right)^{\otimes n}&=\sym\left(\sum_{j\in F}x_j\right)^{\otimes n}
=\sum_{j_1,\dots j_n\in F}\sym(x_{j_1}\otimes\dots\otimes x_{j_n})\\
&=\sum_{\begin{matrix}\supp(I)\subseteq F\\|I|=n\end{matrix}}\begin{pmatrix}n\\ I\end{pmatrix}\sym(\vec{x}^{\otimes I}).
\end{align*}
As a consequence,
\begin{align*}
\sum_{F\subseteq(n)}(-1)^{|F|+n}\left(\sum_{j\in F}x_j\right)^{\otimes n}
&=\sum_{F\subseteq(n)}(-1)^{|F|+n}\sum_{\begin{matrix}\supp(I)\subseteq F\\|I|=n\end{matrix}}\begin{pmatrix}n\\ I\end{pmatrix}\sym(\vec{x}^{\otimes I})\\
&=\sum_{|I|=n}\begin{pmatrix}n\\ I\end{pmatrix}\sym(\vec{x}^{\otimes I})\sum_{\supp(I)\subseteq F\subseteq(n)}(-1)^{|F|+n}.
\end{align*}
We now observe that, setting $j=|\supp(I)|$, the number of sets $F$ of cardinality $\ell$ such that $\supp(I)\subseteq F\subseteq(n)$ is $\begin{pmatrix}n-j\\\ell-j\end{pmatrix}$, hence
\begin{align*}
\sum_{\supp(I)\subseteq F\subseteq(n)}(-1)^{|F|+n}
&=\sum_{\ell=j}^n \begin{pmatrix}n-j\\\ell-j\end{pmatrix}(-1)^{\ell+n}
=\sum_{m=0}^{n-j} \begin{pmatrix}n-j\\m\end{pmatrix}(-1)^{m+j+n}\\
&=(-1)^{j+n}\delta_{jn}=\delta_{jn}.
\end{align*}
Since the only index $I$ with $|I|=n$ and $|\supp(I)|=n$ is $I=(1,\dots,1)$ we get the thesis.
\end{proof}

\begin{lemma}\label{lemma:sa}
Let $a$ be a selfadjoint operator on the one-particle space $\ch$. Then $e^a$ is selfadjoint.
\end{lemma}
\begin{proof}
Let  $e_U$ be the spectral projection of the operator $a$ for the Borel set $U$. Making use of Lemma \ref{lemma:sym}, one can show that vectors of the form $x^{\otimes n}$, $n\in\N$, 
$x\in\cd(a)$, $e_{[-\alpha,\alpha]}x=x$, $\alpha>0$, form a total set in $e^\ch$.
By a direct computation, such vectors are analytic for $e^a$. The thesis follows by Nelson Theorem, \cite{RS2} Theorem X.39.
\end{proof}

\begin{lemma}\label{lemma:n-th-der}
For any vector $h\in \ch$,
$$
\frac{d^n}{dt^n}e^{th}|_{t=0}=\sqrt{n!}\ h^{\otimes n},
$$
where derivatives converge in norm.
\end{lemma}
\begin{proof}
We prove by induction on $n$ that the following formula is true in norm:
$$
\frac{d^n}{dt^n}e^{th}=\bigoplus_{j=0}^{\infty}\frac{\sqrt{(j+n)!}}{j!}t^{j}h^{\otimes (j+n)}.
$$
The result is true for $n=0$. Assume it for $n$, then
\begin{align*}
\frac{d^{(n+1)}}{dt^{(n+1)}}e^{th}&=\lim_{\eps\to0}\eps^{-1}
\bigoplus_{j=0}^{\infty}\frac{\sqrt{(j+n)!}}{j!}((t+\eps)^{j}-t^{j})h^{\otimes (j+n)}\\
&=
\lim_{\eps\to0}\eps^{-1}
\bigoplus_{j=0}^{\infty}\frac{\sqrt{(j+n+1)!}}{(j+1)!}((t+\eps)^{(j+1)}-t^{(j+1)})h^{\otimes (j+n+1)}
\end{align*}
If $|\eps|<1$, 
\begin{align*}
\left|\frac{(t+\eps)^{(j+1)}-t^{(j+1)}}{\eps}-(j+1)t^j\right|
&\leq\sum_{p=2}^{j+1}	
\begin{pmatrix}(j+1)\\p\end{pmatrix}|\eps|^{p-1}|t|^{j+1-p}\\
&\leq |\eps| \sum_{p=0}^{j+1}\begin{pmatrix}(j+1)\\p\end{pmatrix}|t|^{j+1-p}
\leq|\eps|(|t|+1)^{j+1}.
\end{align*}
As a consequence,
\begin{align*}
\|&
\bigoplus_{j=0}^{\infty}\frac{\sqrt{(j+n+1)!}}{(j+1)!}\frac{(t+\eps)^{(j+1)}-t^{(j+1)}}{\eps}h^{\otimes (j+n+1)}
- \bigoplus_{j=0}^{\infty}\frac{\sqrt{(j+n+1)!}}{j!}t^{j}h^{\otimes (j+n+1)}\|^2\\
&= 
\|\bigoplus_{j=0}^{\infty}\frac{\sqrt{(j+n+1)!}}{(j+1)!}h^{\otimes (j+n+1)}
\left(\frac{(t+\eps)^{(j+1)}-t^{(j+1)}}{\eps}-(j+1)t^{j}\right)\|^2\\
&\leq\eps^2
\sum_{j=0}^{\infty}\frac{(j+n+1)!}{((j+1)!)^2}\|h\|^{2 (j+n+1)}(|t|+1)^{2(j+1)}
\end{align*}
\end{proof}

\begin{lemma}\label{lemma:core}
Let  $\Omega$ be a  standard vector for $\car(K)$,  $S$ the associated Tomita operator. Then
\item[$(i)$] If $k\in K$, $Se^{ik}=e^{-ik}$.
\item[$(ii)$] Let $\cd$ be the linear span of the vectors $e^{ik}$, $k\in K$. Then  the closure of $\cd$ w.r.t. the graph norm of $S$ is the domain of $S$.
\item[$(iii)$] For $k_1,\dots,k_n\in K$, $\sym(k_1\otimes\dots\otimes k_n)$ belongs to the domain of $S$, and is invariant under $S$.
\end{lemma}
\begin{proof} $(i)$. If $k\in K$, we have $SW(k)\Omega=W(-k)\Omega$, which implies the thesis.

$(ii)$. Indeed, any operator $A\in\car(K)$ can be written as the limit, in the strong$^*$-topology, of operators  $A_i$, with $A_i$ in the linear span of the $W(k)$'s, $k\in K$, therefore $A_i\Omega$ converges to $A\Omega$ in the graph norm of $S$, i.e. $A\Omega\in\overline{\cd}$. Since $\cd(S)$ is the closure, w.r.t. the graph norm, of $\{A\Omega,A\in\car(K)\}$, the thesis follows.

$(iii)$. By Lemma \ref{lemma:n-th-der}, one gets that, for $k\in K$, $\frac{d^n}{dt^n}e^{itk}|_{t=0}$ is a limit, in the graph norm of $S$, of elements of $\cd$. As a consequence, $k^{\otimes n}\in\cd(S)$, and, because of $(i)$, $Sk^{\otimes n}=k^{\otimes n}$. By Lemma \ref{lemma:sym} one gets that, for $k_1,\dots,k_n\in K$, $\sym(k_1\otimes\dots\otimes k_n)$ belongs to the domain of $S$, and is invariant under $S$. 
\end{proof}

\begin{proof}[Proof of Property $(a)$]
Since  $W(k)\Omega=e^{-\|k\|^2/4}e^{ik/\sqrt2}$, the sets $\{W(k)\Omega:k\in K\}$ and $\{e^{ik}:k\in K\}$ span the same space. From the Lemmas  \ref{lemma:sym} and \ref{lemma:n-th-der} above, the norm closure of the linear span of the set $\{e^k:k\in K\}$ contains all symmetrized elementary tensors of the form $\sym (x_1\otimes\dots\otimes x_n)$, with $x_i\in K$. Therefore, if $K+iK$ is dense, $\Omega$ is cyclic for $\car(K)$.
On the other hand, if $k\in K$ and $k'\in K'$, $\Im (k,k')=0$, therefore the canonical commutation relations imply  $[W(k),W(k')]=0$, i.e. $\car(K')\subset\car(K)'$. Then,  $K\cap iK=\{0\}$ implies, passing to the real-orthogonal complement, $K'+ iK'$ is dense. As a consequence, $\car(K')\Omega$ is dense, hence $\car(K)'\Omega$ is dense, namely $\Omega$ is separating for $\car(K)$. 
\\
Conversely, if $\Omega$ is cyclic for $\car(K)$ the set $\{e^{ik}:k\in K\}$ is total in $e^\ch$. Since the norm closure of the linear span of the symmetrized elementary tensors of the form $\sym (x_1\otimes\dots\otimes x_n)$, $x_i\in K$, contains  $\{e^{ik}:k\in K\}$, the symmetrized elementary tensors above are total in $e^\ch$. In particular, $K$ is total in $\ch$, namely $K+iK$ is dense.
If moreover $\Omega$ is separating, the operator $S$ is defined and, by  Lemma \ref{lemma:core} $(iii)$, $k\in K$ implies $Sk=k$. Then, if $k\in K\cap iK$, we get $Sk=k$ and $Sk=-k$, namely $k=0$
\end{proof}

\begin{proof}[Proof of Property $(b)$]
By Lemma \ref{lemma:core} $(iii)$, one gets $S\supseteq e^s$. On the other hand, by Lemma \ref{lemma:coherentvectors}, $e^{ik}\in\cd(e^s)$, $k\in K$, and $e^{s}$ coincides with $S$ on such vectors. Since the linear span $\cd$ of such vectors is a core for $S$ by Lemma \ref{lemma:core} $(ii)$, we get $e^s\supset S$.
We now observe that, given the polar decomposition $s=j\delta^{1/2}$, $e^j$ is anti-unitary, $e^{\delta^{1/2}}$ is positive selfadjoint by Lemma \ref{lemma:sa}, and $S=e^s=e^je^{\delta^{1/2}}$. Since $S$ is invertible, its polar decomposition $J\Delta^{1/2}$ is uniquely determined by the requirement that $J$ is anti-unitary and $\Delta\geq0$, hence $J=e^j$ and $\Delta=e^{\delta}$.
\\
Finally, $\car(K)'=J\car(K)J=J\{W(k):k\in K\}''J=\{JW(k)J:k\in K\}''=\{W(jk)^*:k\in K\}''=\car(K')$. 
\end{proof}

\subsection{Axioms for Quantum Field Theories}
\subsubsection{Observable algebras \& Haag-Kastler axioms}

We now put together the net of local spaces with the second quantization algebra construction. If $\ch=L^2(H_m,\ d\Omega_m)$, we may consider the following net of von~Neumann algebras on $e^\ch$.
\begin{equation}\label{freefieldlocalg}
\co\to\car(\co):=\car(K(\co)).
\end{equation}
The algebra $\car(\co)$ is interpreted as the algebra whose self-adjoint elements describe the physical quantities which can be observed in the region $\co$. Such net describes non-interacting neutral (i.e. self-adjoint fields) spin-zero Bose particles. In the following we shall only consider causally complete regions, namely regions for which $\co=\co''$, and more specially double cones. A double cone is obtained by applying any Poincar\'e transformation to the causal completion of an open ball in the time-zero plane.
The net $\co\to\car(\co)$, $\co$ being a causally complete region in $M^4$,
satisfies the following properties:

\begin{itemize}
\item[$(1)$] (isotony).\quad
$\co_1\subset\co_2\Rightarrow \car(\co_1)\subset\car(\co_2)$;
\item[$(2)$] (locality).\quad
$\co_1\subset\co_2'\Rightarrow \car(\co_1)\subset\car(\co_2)'$;
\item[$(2')$] (Haag duality).\quad
$ \car(\co')=\car(\co)'$;
\item[$(3)$] (Poincar\'e symmetry).\quad The Poincar\'e group acts as automorphisms of the net,
in such a way that $\alpha_g(\car(\co))=\car(g\co)$, $ g\in \cp^\uparrow_+$. 
\end{itemize}

Properties $(1), (2), (3)$ of a net of von~Neumann algebras (or C$^*$-algebras) on double-cones of the Minkowski space are called Haag-Kastler axioms, and have been proposed as a minimal set of axioms for a local quantum theory (cf.  \cite{HK}). 

A representation $\pi$ of a net $\co\to\car(\co)$ on a Hilbert space $\ch$ is a family $\{\pi_\co\}$, with $\pi_\co$ a representation of $\car(\co)$ on $\ch$, such that $\pi_{\co_2}|_{\car(\co_1)}=\pi_{\co_1}$ if $\co_1\subset\co_2$. When the family of regions is directed, this is the same as giving a representation of the inductive limit C$^*$-algebra.

The net $\co\to\car(\co)$ of local algebras for the free scalar field is equipped with a representation $\pi_0$
satisfying

\begin{itemize}
\item[$(3')$] (Poincar\'e covariance).\quad There exists a strongly continuous unitary representation $U$ of the Poincar\'e group $\cp^\uparrow_+$ such that $U(g)\pi_0(\car(\co))U(g)^*=\pi_0(\car(g\co))$, $ g\in \cp^\uparrow_+$. 

\item[$(4)$] (Positive energy) The joint spectrum of the generators of the translation subgroup lies in the closed forward light  cone.
\item[$(5)$] (vacuum).\quad
There exists a unique (up to a multiplicative constant) translation invariant vector $\Omega$. The set $\{\pi_0(A)\Omega, A\in\car(\co),\co$ double cone$\}$, is dense in $\ch$.
\end{itemize}

A representation $\pi_0$ of a Haag-Kastler net on a Hilbert space is called a {\it vacuum representation} if  properties $(3'), (4), (5)$ above hold.

\begin{proposition} 
With the assumptions above, $U(g)\Omega=\Omega$, for any $g\in\cp^\uparrow_+$.
\end{proposition}
\begin{proof} Let $g$ be an element of the Lorentz group. Then $$U(\tau_x)U(g)\Omega=U(g)U(\tau_{g^{-1}x})\Omega=U(g)\Omega,$$ namely $U(g)\Omega$ is translation invariant. Uniqueness up to a constant imply $U(g)\Omega=c_g\Omega$, with $|c_g|=1$, namely $g\to c_g$ is a one-dimensional representation of $\cl^\uparrow_+$. Since the Lorentz group is perfect, it has no non-trivial one dimensional representations, namely $c_g=1$.
\end{proof}

The following theorem is due to Borchers in this setting, but is usually called Reeh-Schlieder theorem, because of the analogous result in the Wightman setting (see Theorem \ref{RSthm}).

\begin{theorem}[\cite{Borch68}]
If axioms $(1), (2), (3'), (4), (5)$ are satisfied, and additivity holds, namely $\co=\cup_i\co_i$ implies $\car(\co)=\vee_i\car(\co_i)$, then the vacuum vector is cyclic and separating for any double cone.
\end{theorem}
\subsubsection{Free fields \& the G{\aa}rding-Wightman axioms} 
Let us now come back to the free field example. Assume $f$ is in $\cs(M^4,\R)$, $Ef$ is the corresponding element in $L^2(H_m,d\Omega_m)$, and denote by $\phi(f)$ the self-adjoint generator of the one-parameter group $W(\lambda Ef)$, $W(\cdot)$ denoting the Weyl unitary. The map
$$
f\in\cs(M^4)\to \phi(f)
$$ 
is called the free scalar field of mass $m$ The map $\phi$ is usually extended linearly to complex-valued functions.

It satisfies the following properties:

\begin{itemize}
\item[$(A)$] The map $f\to \phi(f)$ is an operator valued tempered distribution.
\item[$(B)$] There is a dense common invariant domain $\cd$ for all fields $\phi(f)$, and $\phi(\overline{f})\subset\phi(f)^*$.
\item[$(B')$] The field operators $\phi(f)$ are essentially self-adjoint on  a dense common invariant domain $\cd$, for real-valued $f$. 
\item[$(C)$] There is a strongly continuous, positive energy unitary representation $U$ of the Poincar\'e group $\cp^\uparrow_+$ satisfying $U(g)\phi(f)U(g)^*=\phi(f_g)$.
\item[$(D)$] There is a unique (up to a multiplicative constant) translation invariant vector $\Omega$, the vacuum vector, contained in $\cd$.
\item[$(E)$] If the supports of $f$ and $g$ are space-like separated, $\phi(f)\phi(g)$ and $\phi(g)\phi(f)$ coincide on $\cd$.
\item[$(E')$] If the supports of $f$ and $g$ are space-like separated, $\phi(f)$ and $\phi(g)$ commute as self-adjoint operators, namely the spectral projections of the former commute with the spectral projections of the latter.
\end{itemize}

Properties $(A), (B), (C), (D)$ and $(E)$ are the so called G{\aa}rding-Wightman axioms for a neutral field of spin zero. They  have been proposed (in the generalized form for charged fields of any spin) as a minimal set of axioms for Quantum Field Theory \cite{SW}.

\begin{remark} 
The relations between the G{\aa}rding-Wightman axioms and the Haag-Kastler axioms have long been investigated. It is not difficult to show that, assuming the extra-axioms $(B')$ and $(E')$, we obtain a net of observable algebras obeying Haag-Kastler axioms, together with the vacuum representation. Conversely there are many papers that tried to recover fields from observables. One problem is that fields are not necessarily observable (they are not gauge-invariant). The problem of the reconstruction of the global gauge group and of the field algebras has been solved by Doplicher-Roberts \cite{DR} in terms of their theory of superselection sectors. The actual reconstruction of fields as operator-valued distributions has been addressed by Fredenhagen and Hertel \cite{FreHa}, Fredenhagen and J\"orss \cite{FreJo}, and Bostelmann\cite{Bostelmann}.
\end{remark}

\section{Bisognano-Wichmann relations.}\label{sec:BW}

The property described by Bisognano and Wichmann in their papers \cite{BW1,BW2} concerns the relation between the modular operators associated with certain space-time regions of the Minkowski space-time in the vacuum representation. We now describe this property in the case of the free scalar field.

\subsection{The Theorem by Bisognano and Wichmann}

\subsubsection{One-particle Bisognano-Wichmann theorem}

Let us consider the so-called right wedge region $W=\{x\in M^4:x_1>|x_0|\}$, and observe that such region is invariant for the one-parameter subgroup $\Lambda_W(t)$ of the Lorentz group
$$
\Lambda_W(t)=\begin{pmatrix}
\cosh (2\pi t)&-\sinh(2\pi t)& 0&0\\
-\sinh (2\pi t)&\cosh (2\pi t)& 0&0\\
0&0&1&0\\
0&0&0&1
\end{pmatrix}
$$

\begin{theorem}\label{thm:1pBW}
Let $K(W)$ be the closed real subspace of the one-particle space for the free scalar field, corresponding to the wedge region $W$.
The modular group and conjugation associated with the space $K(W)$ have a geometric action. More precisely
\begin{equation}\label{1pBW}
j_W={\uni}(r_1,0),\qquad \delta^{it}={\uni}(\Lambda_W(t),0),
\end{equation}
where $r_1$ denotes the (proper, time-reversing) transformation changing sign to the $x_0$ and $x_1$ coordinates.
Making use of the one-particle PCT transformation $\theta$, one may also write $j_W=\theta\cdot {\uni}(R_{23}(\pi),0)$, where $R_{23}$ denotes a rotation on the $(x_2,x_3)-$plane.
\end{theorem}
\subsubsection{The general case}

Assume $\phi$ is a neutral scalar field satisfying the G{\aa}rding-Wightman axioms in the stronger form
$(A), (B'), (C), (D)$ and $(E')$, and assume also that the field is {\it irreducible}, namely nothing but multiples of the identity commutes with all fields.

For any open region $\co$,  denote by $\car(\co)$ the von~Neumann algebra generated by bounded functional calculi of the fields $\phi(f)$, with supp$f\subset\co$. Then:

\begin{theorem}[Reeh-Schlieder Theorem]\label{RSthm}
For any non-empty open region $\co$, the vacuum vector $\Omega$ is cyclic for the algebra $\car(\co)$.
In particular, if $\co'$ is non-empty, $\Omega$ is a standard vector for $\car(\co)$.
\end{theorem}
For a proof, see e.g. \cite{SW}.

\begin{theorem}[Bisognano-Wichmann Theorem \cite{BW1}]\label{thm:BW}
If $J_W$, $\Delta_W$ denote the modular operators for the pair $(\car(W),\Omega)$, then 
\begin{equation}\label{BWrel}
J_W=\Theta\cdot U(R_{23}(\pi),0),\qquad \Delta^{it}=U(\Lambda_W(t),0).
\end{equation}
Moreover, wedge duality holds, namely $\car(W)'=\car(W')$.
\end{theorem}
A somewhat simpler proof is contained in \cite{Rigotti}.
The generalization to non-necessarily scalar fields was given in \cite{BW2}.

\medskip
Up to now, the only region for which  the modular objects have been proved to have a geometrical meaning is the right wedge. However,  the vacuum is invariant under  the Poincar\'e group. Therefore, denoting by $W_R$ the right wedge, the modular operator $S_{gW_R}$ for the pair $(\car(gW_R),\Omega)$ coincides with $U(g)S_{W_R}U(g)^*$. We then call wedge any element of the set $\cw:=\{gW_R,g\in\cp^\uparrow_+\}$, and observe that for any $W\in\cw$ the corresponding modular operators have a geometrical meaning; setting $\Lambda_{gW_R}(t)=g\Lambda_{W_R}(t)g^{-1}$, $r_{gW_R}=gr_{W_R}g^{-1}$, we get $J_W=U(r_W)$, $\Delta_W^{it}=U(\Lambda_W(t))$.

Given a net $\co\to\car(\co)$, $\co$ double cone, we may set $\car(\cc)=\vee_{\co\subset\cc}\car(\co)$ for a general open region $\cc$. Then the dual net is defined by $\car^d(\co)=\car(\co')'$. The dual net does not necessarily  satisfy locality, but if it does, it clearly satisfies Haag duality:
$\car(\co)\subset\car(\co')'$ implies $\car^d(\co')'=\car(\co)\subset\car(\co')'=\car^d(\co)$, on the other hand locality for $\car^d(\co)$ gives the reverse inclusion. The net $\co\to\car(\co)$ is said to satisfy {\it essential duality} if $\co\to\car^d(\co)$ is local (hence dual) for double cones. The following result can be found e.g. in \cite{Rigotti}.

\begin{corollary} \label{essdual}
With the assumptions of the present section, the net $\co\to\car(\co)$ satisfies essential duality.
\end{corollary}

\begin{proof}
We first note that the two a-priori different notions for $\car(W)$ actually coincide, namely the von~Neumann algebra generated by fields localized in $W$ coincides with the algebra $\vee_{\co\subset W}\car(\co)$. Indeed, $\vee_{\co\subset W}\car(\co)$ is contained in $\car(W)$, and is globally invariant under the action of the modular group of $\car(W)$, since the group acts geometrically.
By a theorem of Takesaki, we get a conditional expectation from $\car(W)$ to the invariant subalgebra. Since the vacuum is cyclic for the latter, the conditional expectation is the identity, namely the two algebras  coincide. 

Let us observe that if  $\co$ and $\co_0$ are space-like separated double cones, there exists a wedge $W$ such that $\co\subset W\subset\co_0'$. Then, for any double cone $\co_0$ we get, by wedge duality,
\begin{align*}
\car^d(\co_0)&=\car(\co_0')'=\left(\bigvee_{\co\subset\co_0'}\car(\co)\right)'
=\left(\bigvee_{W\subset\co_0'}\car(W)\right)'\\
&=\bigwedge_{W\subset\co_0'}\car(W)'
=\bigwedge_{W\supset\co_0}\car(W).
\end{align*}
Finally, if  $\co_0$ and $\co_1$ are space-like separated double cones, we find a wedge $W$ such that $\co_0\subset W$ and $\co_1\subset W'$, hence the corresponding algebras $\car^d(\co_0)$, $\car^d(\co_1)$ commute, namely $\car^d$ is local.
\end{proof}

\begin {remark}
A striking result in quantum field theory is the relation between the statistical behavior of quantum particles, which may be described either by the Bose-Einstein statistics or by the Fermi-Dirac statistics, and is manifested by either the commutation or the anticommutation relations for fields at space-like distance, and the integer or half-integer values for the spin, corresponding to the symmetry group being truly represented, or represented up to a phase, namely the appearance of a representation of the symmetry group or of its universal covering. In the Wightmann framework the proof of this connection follows by the implementability of the PCT symmetry by the operator $\Theta$ (the so called PCT theorem, cf. \cite{SW}). As shown by the Bisognano-Wichmann theorem, such PCT operator is related to the modular conjugation $J_W$.  In the algebraic setting, the geometrical meaning of $J_W$ is  indeed the base for a proof of the connection between spin and statistics (cf. \cite {GuLo2,GuLo3,KuckSpin}).
\end{remark}
\subsubsection{The conformal case}

As shown above, the larger is the symmetry group the larger is the family of regions for which the modular objects have a geometric meaning. This observation produced an important result of Hislop and Longo.

Let us recall that on a (semi)-Riemannian manifold, a conformal transformation is a transformation which preserves the metric tensor up to a scalar function. In dimension $\geq3$, the conformal group is a finite dimensional Lie group; for the space $M^4$ its identity component is generated by the proper Poincar\'e group and the relativistic ray inversion transformation $x\to\frac{x}{x\cdot x}$. A quantum field theory on $M^4$ is conformal if the identity component of the conformal group acts as the symmetry group. We note here that conformal transformations are singular on some submanifolds. The way to treat this problem is to extend the theory to a suitable compactification of the space-time (or better to its universal covering). For a detailed description of this procedure see \cite{BGL1}.

\begin{theorem}[Hislop-Longo \cite{HL}]
Assume $\phi$ is an irreducible neutral scalar field satisfying the G{\aa}rding-Wightman axioms in the stronger form
$(A), (B'), (C'), (D')$ and $(E')$, where
\item[$(C')$] There is a strongly continuous, positive energy unitary representation $U$ of the conformal group  satisfying $U(g)\phi(f)U(g)^*=\phi(f_g)$.
\item[$(D')$] There is a unique (up to a multiplicative constant) translation invariant vector $\Omega$, the vacuum vector, contained in $\cd$, which is also conformally invariant.

Then the modular conjugations and groups of the von~Neumann algebras associated with wedges, double cones, and forward and backward light cones have a geometric meaning, namely for any such region $\cc$ there is a conformal reflection $r_\cc$ and a one parameter group $\Lambda_\cc$ of conformal transformations such that
\begin{equation}\label{confBWrel}
J_\cc=U(r_\cc),\qquad \Delta_\cc^{it}=U(\Lambda_\cc(t)).
\end{equation}
\end{theorem}

\begin{proof} [Sketch of the proof]
It is enough to show that double cones and forward and backward cones may be obtained by wedges via conformal transformation. Indeed, applying the  relativistic ray inversion transformation to the right wedge translated by the vector $(0,r/2,0,0)$ one gets the double cone whose basis is the ball of radius $1/r$ and center $(0,-1/r,0,0)$ in the hyperplane $x_0=0$. All other double cones are obtained via Poincar\'e transformations. Applying the ray inversion transformation to the double cone  whose basis is the ball of radius $r$ and center $(r,0,0,0)$ in the hyperplane $x_0=r$ one gets the forward cone based on the point $(1/2r,0,0,0)$. All other cones can be obtained via translations and space-time reflections.
\end{proof}

\subsubsection{The case of the de Sitter space}\label{desitter} 
Instead of changing only the symmetry group, one can change the spacetime itself, and study the geometrical meaning of modular operators on different spacetimes. The four-dimensional de Sitter space may be seen as the hyperboloid $\{(t,\vec{x})\in M^5:\vec{x}^2=R^2+t^2\}$ in the five-dimensional Minkowski space-time. The intersections of the hyperboloid with the wedges of $M^5$ whose edge contains the origin play the r\^ole of the wedges, and the Lorentz group plays the r\^ole of the symmetry group. An analog of the Bisognano-Wichmann theorem for the de Sitter space has been proven in \cite{BrEpMo}, where the  spectrum condition, which is unavailable in de Sitter space since there are no translations, is replaced by analyticity properties of the $n$-point functions.

\subsection{Borchers theorem and BW relations for conformal nets}
For many years, the Theorem by Bisognano and Wichmann was a kind of paradoxical result: while its formulation is very natural in the Haag-Kastler formalism, its proof was given only in the G{\aa}rding-Wightman setting. The first main result in the algebraic formalism is due to Borchers, and is of a quite abstract nature.

\begin{theorem}[Borchers \cite{Borch92}]\label{Borch92}
Let $\car$ be a von~Neumann algebra with a standard vector $\Omega$, and $U(a)$ a one-parameter group of unitaries with positive generator  leaving $\Omega$ fixed and such that, for $a\geq0$, $U(a)\car U(a)^*\subset\car$. The the following commutation relations between the modular operators $\Delta$ and $J$ for $(\car,\Omega)$ and $U(a)$ hold:

\begin{equation}
\Delta^{it}U(a)\Delta^{-it}=U(e^{-2\pi t}a),\quad JU(a)J=U(-a),\quad t,a\in\R.
\end{equation}

\end{theorem}
The following proof is due to Florig \cite{Florig}.
\begin{proof}
Set $V(a)=JU(-a)J$. 
Let us observe that $V(a)\Omega=JU(-a)J\Omega=\Omega$, and, if $x\in\car, x'\in\car'$, $a\geq0$,
\begin{align*}
[V(a)xV(a)^*,x']
&=V(a)[x,V(a)^*x'V(a)]V(a)^*\\
&=V(a)[x,JU(a)(J x'J)U(a)^*J]V(a)^*=0,
\end{align*}
since $(J x'J)\in\car$, and $U(a)$ implements endomorphisms of $\car$ for $a\geq0$. Moreover, if $H$ denotes the selfadjoint generator of the one-parameter group $U$, the selfadjoint generator of $V$ is given by
$$
-i\frac{d}{da}V(a)=-i\frac{d}{da}Je^{-iaH}J=-iJ(-iH)J=JHJ,
$$
namely $V$ has a positive generator, too. This shows that $V$ has the same properties of $U$.

Now, for $x\in\car$, $x'\in\car'$, $0\leq\Im z\leq\frac12$, set
$$
f_U(z)=(\Delta^{-i\overline{z}}x'\Omega,U(e^{2\pi z}a)\Delta^{-iz}x\Omega),\qquad a\geq0.
$$
Let us prove that $f$ is continuous and bounded in the closed strip $0\leq\Im z\leq\frac12$, and analytic in the open strip $0<\Im z<\frac12$. Indeed, $e^{-2\pi z}a$ belongs to the upper half-plane, where $U$, having a positive generator, is analytic. Moreover, 
\begin{align*}
\Delta^{-iz}x\Omega
&=\left(\Delta|_{[01]}^{-iz}+\Delta|_{(1,\infty)}^{-iz-1/2}\Delta|_{(1,\infty)}^{1/2}\right)x\Omega\\
&=\Delta|_{[01]}^{-iz}x\Omega + \Delta|_{(1,\infty)}^{-iz-1/2}\Delta^{1/2}x\Omega,
\end{align*}
where $\Delta|_{E}=\Delta P_\Delta(E)$, with $P_\Delta(E)$ the spectral projection of $\Delta$ relative to the measurable set $E$. Now, since $\Delta|_{[01]}\leq I$, 
$\Delta|_{[01]}^{w}$ is analytic for $\Re w>0$, namely 
$\Delta|_{[01]}^{-iz}$ is analytic for $\Im z>0$. Analogously, since $\Delta|_{(1,\infty)}\geq I$, 
$\Delta|_{(1,\infty)}^w$ is analytic for $\Re w<0$, namely $\Delta|_{(1,\infty)}^{-iz-1/2}$ is analytic for $\Im z<1/2$. The analyticity of $\Delta^{-iz}x\Omega$ in the open strip follows, since $x\Omega\in\cd(\Delta^{1/2})$. The same argument shows that in the closed strip $\|\Delta|_{[01]}^{-iz}\|\leq1$ and $\|\Delta|_{(1,\infty)}^{-iz-1/2}\|\leq1$, hence $\|\Delta^{-iz}x\Omega\|\leq\|x\Omega\|+\|\Delta^{1/2}x\Omega\|$. In an analogous way we get analyticity and boundedness for $\Delta^{-i\overline{z}}x'\Omega$.

Since $V$ has the same properties as $U$, we get that $f_V(z)$ is continuous and bounded in the closed strip $0\leq\Im z\leq\frac12$, and analytic in the open strip $0<\Im z<\frac12$.

We now show that $f_U(t+\frac{i}2)=f_V(t)$, $\forall t\in\R, a\geq0$. Indeed
\begin{align*}
f_U(t+\frac{i}2)
&=(\Delta^{-\frac12}\Delta^{-it}x'\Omega,U(e^{2\pi t}e^{i\pi }a)\Delta^{-it}\Delta^{\frac12}x\Omega)\\
&=(\Delta^{-\frac12}\Delta^{-it}x'\Omega,U(-e^{2\pi t}a)\Delta^{-it}Jx^*\Omega)\\
&=(\Delta^{-\frac12}\Delta^{-it}x'\Omega,JV(e^{2\pi t}a)\Delta^{-it}x^*\Delta^{it}V(e^{2\pi t}a)^*\Omega)\\
&=(\Delta^{-it}x'\Omega,SV(e^{2\pi t}a)\Delta^{-it}x^*\Delta^{it}V(e^{2\pi t}a)^*\Omega)\\
&=(\Delta^{-it}x'\Omega,V(e^{2\pi t}a)\Delta^{-it}x\Omega)=f_V(t),
\end{align*}
where in the third (and fifth) equation we used the invariance of $\Omega$ under  $V(a)$ and $\Delta^{it}$, and in the fifth one we used the fact that $V(e^{2\pi z}a)\Delta^{-it}x^*\Delta^{it}V(e^{2\pi t}a)^*\in\car$, since $\Delta^{it}$ implements automorphisms of $\car$ for $t\in\R$, and $V(a)$ implements endomorphisms of $\car$ for $a\geq0$.

In the same way, $f_V(t+\frac{i}2)=f_U(t)$, since the map $U(a)\to JU(-a)J$ is an involution. As a consequence, gluing copies of the functions $f_U$ and $f_V$ we get a continuous function $f$ on the complex plane which is periodic of period $i$, satisfies $f(z)=f_U(z-i\frac{m}2)$ on any closed strip $\frac{m}2\leq\Im z\leq \frac{m+1}2$ for $m$ even and  satisfies $f(z)=f_V(z-i\frac{m}2)$ on any closed strip $\frac{m}2\leq\Im z\leq \frac{m+1}2$ for $m$ odd. As a consequence, by the edge of the wedge theorem (cf. e.g. \cite{SW}), $f$ is analytic on $\C$, and, being bounded, is indeed constant by Liouville theorem. This entails
\begin{align*}
&(\Delta^{-it}x'\Omega,U(e^{2\pi t}a)\Delta^{-it}x\Omega)=f(t)=f(0)=(x'\Omega,U(a)x\Omega),\\
&(x'\Omega,U(a)x\Omega)=f(0)=f(i/2)=f_V(0)=(x'\Omega,JU(-a)Jx\Omega).
\end{align*}
Since $\Omega$ is cyclic and separating for $\car$, we get
\begin{equation}\label{halfrel}
\begin{aligned}
U(a)&=\Delta^{it}U(e^{2\pi t}a)\Delta^{-it},\\
U(a)&=JU(-a)J
\end{aligned}
\quad a\geq0, \ t\in\R.
\end{equation}
The same relations should hold for $V$, therefore we get (\ref{halfrel}) for $a\leq0$.
\end{proof}

The previous theorem has a direct corollary for two-dimensional quantum field theories in the algebraic setting, which motivated Borchers theorem.

\begin{corollary}[\cite{Borch92}]
Assume we are given a net $\co\to\car(\co)$ of von~Neumann algebras, where $\co$ is a double cone in the two-dimensional Minkowski space $M^2$, acting in the vacuum representation, so that axioms
$(1), (2), (4), (5)$ are satisfied, with axiom $(3')$ replaced by
\item[$(3'')$] There exists a strongly continuous, positive energy, unitary representation $U$ of the translation group such that
 $U(\tau_x)\pi_0(\car(\co))U(\tau_x)^*=\pi_0(\car(\co+x))$, where $\tau_x$ denotes the translation by the vector $x$.

Then the representation $U$ extends to a representation of the proper Poincar\'e group in such a way that $(3')$ is satisfied for the dual net $\car^d(\co)$ and Bisognano-Wichmann relations hold. In particular essential duality holds for the given net.
\end{corollary}

\begin{remark} In the previous Corollary Poincar\'e covariance is not assumed, indeed the Lorentz boosts are constructed via modular groups. More precisely what is proved is a geometric action of the modular groups and reflections (see next section). This is stronger than the Bisognano-Wichmann result, however rises the question of the uniqueness of the implementation of the Poincar\'e symmetry. It may happen that the theory was endowed with a Poincar\'e action which does not coincide with the one recovered by modular theory. A uniqueness result is contained in \cite{BGL1},
a comprehensive review on these questions is given in \cite{BoYn}.
\end{remark}

\begin{remark}
It is possible to reverse the statement of the previous Corollary, namely reconstruct the net of local algebras starting with the vacuum vector $\Omega$, the right wedge algebra $\cam$ and a positive energy representation $U$ of the translation group on  the two-dimensional Minkowski space $M^2$. The standard hypotheses are required, namely $\Omega$ should be cyclic and separating for $\cam$ and invariant under $U$, and $U(x)$ has to implement endomorphisms of $\cam$ when $x$ is space-like and pointing to the right. In this way one reconstructs the algebras $\car(W)$ for all wedges together with a representation of the Poincar\'e group in two dimensions (cf. the reconstrution results for conformal theories on the circle in Theorem \ref{confreconstruction}). Then one may define the double cone algebras via intersection (cf. eq. (\ref{convex}) below). However, the non triviality of the double cone algebras is not guaranteed in general. This problem has been solved under the further assumption of modular nuclearity, bringing to the construction of interacting theories (see \cite{Lechner} and references therein).
\end{remark}

The results by Hislop-Longo and the theorem of Borchers have been used to get the following.
\begin{theorem}[\cite{BGL1},\cite{GaFr}]\label{confBW}
Let $\co\to\car(\co)$ a conformally covariant net of von~Neumann algebras acting on a Hilbert space. Then it satisfies essential duality, and  the modular conjugations and groups of the von~Neumann algebras associated with wedges, double cones, and forward and backward cones have a geometric meaning. In particular, for any such region $\cc$, the relations in $(\ref{confBWrel})$ are satisfied.\end{theorem}

The previous result is the first 4-dimensional example of a complete proof of the Bisognano-Wichmann relations in the algebraic setting, though under the conformal symmetry assumption. 
A proof for massive Poincar\'e covariant theories is due to Mund \cite{Mund}.

\subsection{Physical interpretations: Hawking radiation \& Unruh effect}

In the early seventies, Haag, Hugenholtz and Winnink \cite{HHK} showed that the analyticity condition enjoyed by the modular group was indeed equivalent to the so-called KMS condition in quantum thermodynamics, characterizing equilibrium states  for a given time evolution. According to this interpretation, the Bisognano-Wichmann relations for wedge-like regions mean that the vacuum state is a thermal equilibrium state for the time evolution given by the Lorentz boosts. Indeed, as observed by Sewell \cite{Sewell}, an observer whose time-evolution is given by Lorentz boosts is a uniformly accelerated observer, and, by the Einstein equivalence principle, behaves like a free-falling observer in a gravitational field. The wedge region, as a space-time in itself, is known as the Rindler wedge, and is one of the space-times describing a black hole, the wedge boundary describing the event horizon. A fundamental result of Hawking \cite{Haw} showed that free falling observers in a black hole feel a temperature, the so called Hawking temperature. A heuristic explanation is the following: spontaneous
particle pairs creation happens on the event horizon, negative energy
particles may tunnel into the inaccessible region, the others
contribute to the thermal radiation. This explains why the vacuum becomes a thermal state for an accelerated observer. The general fact that a vacuum state may become a temperature state because of acceleration is generally called Unruh effect \cite{U}. Furthermore, the width of the analyticity strip associated with the KMS state is interpreted as the inverse temperature, hence, according to the re-parametrization of the boosts, for a uniformly accelerated observer with acceleration $a$, the vacuum has temperature $\frac{a}{2\pi}$.

\medskip

A similar motivation explains Bisognano-Wichmann relations for the de Sitter space-time. Gibbons and Hawking  \cite{GH} have shown that a spacetime $\cam$ with repulsive (i.e. positive)
cosmological constant has certain similarities with a black hole
spacetime.  $\cam$ is expanding so rapidly that, if $\gamma$ is a
freely falling observer in $\cam$, there are regions of $\cam$ that
are inaccessible to $\gamma$, even if he waits indefinitely long; in
other words the past of the world line of $\gamma$ is a proper
subregion $\cn$ of $\cam$.  The boundary $\mathfrak H$ of $\cn$
is a cosmological event horizon for $\gamma$.  
As in the black hole case, one argues that $\gamma$ detects a temperature related to the
surface gravity of $\mathfrak H$.  A heuristic explanation can be given as above, the event horizon being replaced by the cosmological horizon.

\section{Modular covariance and modular localization.}\label{GMA&ML}

Once the geometric meaning of the modular objects in quantum field theory has been established, one may try to reverse the procedure, namely to start with modular objects and reconstruct some aspects of a quantum field theory.

\subsection{Modular covariance}

In the spirit of Theorem \ref{Borch92} of Borchers and in fact based on it, one may ask if the modular conjugations or one-parameter groups associated to wedge regions generate a representation of the Poincar\'e group which acts geometrically on the net. So, instead of assuming Poincar\'e symmetry one may try to recover it by modular objects. A hypothesis which is sufficient for that is the request that the adjoint action of the modular groups maps local algebras to local algebras in some prescribed way. These assumptions have been called modular covariance or geometrical modular action. The first result in this direction is the following \cite{BGL2,GuLo2}

\begin{theorem}
Let $\co\to\car({\co})$ a net of von~Neumann algebras acting on a Hilbert space $\ch$ and satisfying the following properties:
\begin{itemize}
\item (isotony).\quad
$\co_1\subset\co_2\Rightarrow \car(\co_1)\subset\car(\co_2)$;
\item (locality).\quad
$\co_1\subset\co_2'\Rightarrow \car(\co_1)\subset\car(\co_2)'$;
\item (Reeh-Schlieder property).\quad
There exists a vector $\Omega$ which is ciclic for the algebras $\car(\co)$ associated with double cones.
\item (modular covariance) For any wedge $W$ and any double cone $\co$, we have
$$
\Delta_W^{it}\car(\co)\Delta_W^{-it}=\car(\Lambda_W(t)\co),\quad t\in\R.
$$
\end{itemize}
Then there exists a unique positive-energy anti-unitary representation $U$ of the proper Poincar\'e group such that $U(g)\car(\co)U(g)^*=\car(g\co)$, and the Bisognano-Wichmann relations hold. In particular the net satisfies essential duality.

If we assume essential duality, the thesis still holds if we replace modular covariance with the less restrictive assumption
$$
\alpha^W_t(\car(\co))=\car(\Lambda_W(t)\co),\quad t\in\R, \co\subset W,
$$
where $\alpha^W_t$ denotes the modular automorphism group associated with $(\car(W),\Omega)$.
\end{theorem}

Many results of this type have been given later, among which we quote \cite{Bo1,Bo2,BFS,BDFS,BuSu2,BuSu3}. In particular, Buchholz, Dreyer, Florig and Summers proved a quite general result, which we state here in the classical case of wedge algebras of the Minkowski space.

A map $W\to\car(W)$ from wedges to von~Neumann algebras satisfies the Condition of Geometric Modular Action (CGMA) if $(i)$ preserves inclusion; $(ii)$ if ${W_1}\cap{W_2}\ne\emptyset$, the vacuum is cyclic for $\car(W_1)\cap\car(W_2)$; $(iii)$ if the vacuum is cyclic for $\car(W_1)\cap\car(W_2)$, then $\overline{W_1}\cap\overline{W_2}\ne\emptyset$; $(iv)$ for any wedge $W_0$, $J_{W_0}$ implements an automorphism of the family $\{\car(W)\}$.

The map $W\to\car(W)$ satisfies the Condition of Geometric Action for the Modular Group (CMG) if $(i)$, $(ii)$, $(iii)$ above are satisfied, and $(v)$ for any wedge $W_0$, $t\in\R$, $\Delta^{it}_{W_0}$ implements an automorphism of the family $\{\car(W)\}$.

The map $W\to\car(W)$ satisfies the Modular Stability Condition if the modular unitaries of any wedge are contained in the group generated by the modular conjugations of all wedges.

\begin{theorem}[\cite{BDFS,BuSu2}]
Assume we have a map $W\to\car(W)$ from wedge regions to von Neumann algebras acting on a given Hilbert space such that CGMA is satisfied, and the group generated by the $J_W$'s acts transitively on the wedge algebras. Then there is a (anti)-unitary representation of the proper Poincar\'e group such that $J_W=U(r_W)$ and $U(g)\car(W)U(g)^*=\car(gW)$. In particular wedge duality holds.

If we assume modular stability then 
the spectrum condition holds (up to a sign).

If we assume CMG with transitive action we get a strongly continuous unitary representation of the  covering group of the proper orthochronous Poincar\'e group such that $U(g)\car(W)U(g)^*=\car(gW)$ and $\Delta_W^{it}=U(\Lambda_W(t))$. 

If locality is further assumed, the representation extends to the proper Poincar\'e group, $J_W=U(r_w)$ and the spectrum condition holds. 
\end{theorem}

Let us remark that the result just quoted was indeed stated and proved in a far more general setting, in order to treat general space-times and general symmetry groups, via suitable family of "wedge like" regions.

A direct analysis of the de Sitter space in terms of a suitable thermal behavior of the vacuum vector is contained in \cite{BoBu}.

We conclude this subsection by mentioning some important results due to Borchers and Wiesbrock, of particular interest in the case of conformal theories.

As already mentioned, for conformally covariant theories Bisognano-Wichmann relations hold (cf. Theorem \ref{confBW}). Chiral two dimensional theories split in a tensor product of two nets on a line, and any such net is covariant w.r.t. the group of fractional linear transformations (M\"{o}bius group). Therefore, any such net extend to a net on the one-point compactification, namely a net on the circle.

Such nets are given by a map $I\to\car(I)$ from the open intervals of the circle to von Neumann algebras acting on a Hilbert space, verifying the suitably modified Haag-Kastler axioms, with the space-like complement replaced by the interior of the complement, the symmetry group replaced by the M\"{o}bius group, and the spectrum condition by the positivity of the generator of rotations (cf. e.g. \cite{Longobook} for a thorough description of the properties of conformal nets on the circle).

Under these hypotheses $\Omega$ is cyclic and separating for the algebras $\car(I)$, and the Bisognano-Wichmann relations hold, namely $\Delta_I^{it}=U(\Lambda_I(t))$, where $\Lambda_I(t)$ is the unique one-parameter subgroup of the M\"obius group leaving $I$ globally invariant, the re-scaling being fixed by the KMS condition, or, equivalently, by the Borchers relations.

As already mentioned, we may also recover the M\"obius symmetry assuming modular covariance for the modular groups plus Reeh-Schlieder property.

With the following result, Wiesbrock showed that one may reconstruct the covariant net itself starting with few algebras. Indeed, assume we have a conformal net on the line, and consider the algebras $\cam:=\car(-1,1)$ and $\cn:=\car(0,\infty)$. It turns out that $J_\cam$ implements the ray inversion map, hence preserves $\cn$,
therefore 
\begin{equation}\label{HWW1}
J_\cam\cn J_\cam=\cn.
\end{equation} 
 $\Delta_\cam^{it}$ 
moves points in $(-1,1)$ toward 1 for negative $t$, 
 therefore, since $\cam\cap\cn=\car(0,1)$, 
\begin{equation}\label{HWW2}
\Delta_\cam^{it}(\cam\cap\cn)\Delta_\cam^{-it}\subset\cam\cap\cn,\quad t\leq0.
\end{equation} 
Finally,  $\Delta_\cn^{it}$ 
 implements contractions for positive $t$, 
therefore
\begin{equation}\label{HWW3}
\Delta_\cn^{it}(\cam\cap\cn)\Delta_\cn^{-it}\subset\cam\cap\cn,\quad t\geq0.
\end{equation} 

\begin{theorem}[\cite{Wies1}]\label{confreconstruction}
Assume we have two von~Neumann algebras acting on a Hilbert space $\ch$ with a common cyclic and separating vector $\Omega$. If relations (\ref{HWW1}), (\ref{HWW2}), (\ref{HWW3}), hold, these data reconstruct in a canonical way a M\"{o}bius covariant net of von~Neumann algebras on the circle.
\end{theorem}

Let me just remark that the proof of previous theorem was based on a result of the same author, whose proof however had a gap. Such gap was filled by a subsequent paper by Araki-Zsido \cite{ArZs}, cf. also \cite{Florig}.

Similar techniques have been used by Wiesbrock and collaborators to recover Poincar\'e covariant nets on the 3-dimensional and 4-dimensional Minkowski space from few algebras with suitable relations \cite{Wies2,KaeWie}.

\subsection{Modular localization}\label{modloc}

This last section is devoted to the observation that, in the case of free fields, the whole net may be reconstructed via the representation of the symmetry group, in terms of one-particle modular operators.
For related results and applications we refer to \cite{FaSc} and references quoted therein.

As we have seen in section \ref{locstruc}, in the one-particle Hilbert space we have a map $\co\to K(\co)$ which associates a standard vector space $K(\co)$ with any double cone $\co$. Since $K(\co)$ is standard, we may define its modular operators $J_\co$, $\Delta_\co$. When wedge regions are considered, Bisognano-Wichmann relations (\ref{1pBW}) hold. 

Now suppose a strongly continuous, (anti)-unitary, positive energy representation ${\uni}$ of the proper Poincar\'e group is given, acting on a Hilbert space $\ch$. We may use Bisognano-Wichmann relations as a prescription: first we associate with any wedge region $W$ the corresponding reflection $r_W$ and one-parameter group $\Lambda_W(t)$, then set $j_W={\uni}(r_W)$, $\delta_W^{it}={\uni}(\Lambda_W(t))$.

With this definition, $j_W$ is anti-linear and commutes with $\delta_W^{it}$, therefore it anti-commutes with $\delta_W^{1/2}$. Hence, setting $s_W=j_W\delta^{1/2}$ we get an anti-linear closed operator such that $s_W^2\subset I$. We may therefore set
\begin{equation}\label{modSsp}
K_W:=\{h\in\cd(s_W):s_Wh=h\}.
\end{equation}

The following result holds.

\begin{theorem}[\cite{BGL3}]
Assume we are given a strongly continuous, (anti)-unitary, positive energy representation ${\uni}$ of the proper Poincar\'e group, acting on a Hilbert space $\ch$.
For any wedge $W$, the space $K(W)$ constructed as above is standard, with Tomita operator $s_W$. The family $W\to K(W)$ satisfies the following properties:
\begin{itemize}
\item[(isotony)]\quad
$W_1\subset W_2\Rightarrow K(W_1)\subset K(W_2)$;
\item[(duality)]\quad
$ K(W')=K(W)'$;
\item [(covariance)]\quad ${\uni}(g)K(W)=K(gW)$, $ g\in \cp_+$. 
\end{itemize}
By definition, the Bisognano-Wichmann relations hold.
\end{theorem}

Let us observe that the properties of covariance and Bisognano-Wichmann are  quite direct consequences of the definitions. As for duality, it follows from the relations $J_W=J_{W'}$ and $\Delta_W^{it}=\Delta_{W'}^{-it}$. Finally, it is sufficient to prove isotony for the inclusion $\tau_x W_R\subset W_R$, where $\tau_x$ is a translation and $x=(t,t,0,0)$, $t\geq0$. The latter is a partial converse of the (one particle) Borchers theorem, the commutation relations, positivity of the generator of translations and cyclicity of the vacuum being assumed, and the endomorphic action of translations to be proved. We refer to \cite{BGL3} or to the recent general notes of Longo \cite{Longobook} for a proof.

\medskip

The net $W\to K(W)$ may be extended to general causally complete regions as follows:
\begin{align}
K(\co)&:=\bigcap_{W\supset\co}K(W)\quad \co {\rm\ convex},\label{convex}\\
K(\cc)&:=\bigcup_{\co\subset\cc}K(\co)\quad \co {\rm\ convex,}\ \cc {\rm\ not\ convex.}\label{nonconv}
\end{align}
Clearly the properties of isotony, locality, and covariance hold for this extended net. Duality for double cones holds too, as in Corollary \ref{essdual}. The non-trivial property is the standard property for $K(\co)$. One may anyway use the second quantization functor and construct the algebras on the Fock space as above: $\car(\co)=\{W(h):h\in K(\co)\}''$.

In fact, Theorem \ref{thm:1pBW} shows that the usual free fields may be alternatively reconstructed via {\it modular localization}. However, not all irreducible positive energy representations of the Poincar\'e group are considered in physics. A subfamily of them, called infinite spin representations, have always been considered as non-physical, cf. e.g. \cite{Yng1}, where it is shown that the construction of free fields associated with these representations is not allowed. The procedure described above however allows the construction of a free field net for these representations. Indeed, for infinite spin representation, it is conjectured that $K(\co)$ is not standard for double cones. However the following holds.

\begin{theorem}[\cite{BGL3}]
Let $\cc$ be a  {\em space-like cone}, namely a convex cone generated by a point and a double cone which are space-like separated.

With the assumptions above, the space $K(\cc)$ is standard. If ${\uni}$ does not contain the trivial representation, the second quantization algebra $\car(\cc)$ is a type III$_{1}$ factor.
\end{theorem}

 The quoted results, in particular the idea that for infinite spin representations local algebras for bounded regions may be trivial, but those for arbitrarily thin cones are not, suggested the construction of string localized fields for infinite spin representations \cite{MSY}. 

An interesting aspect of the previous construction is the possibility of generalizing it to more general space-times or symmetry groups. The main point is to identify a suitable family of regions endowed with a reflection and a one-parameter group with prescribed properties. One example is given by conformal theories: in this case the family is that of double cones; given any representation of the conformal group  for $M^d$ one may construct the associated free field algebras as above. The second example is the de Sitter space: in this case the family consists of the wedges considered in section \ref{desitter}; given a representation of the Lorentz group one may e.g. reconstruct the free field algebras described by Bros and Moschella \cite{BrMo}. While the method in \cite{BGL3} for the de Sitter space applies to all representations of the Lorentz group, the standard property for bounded regions was proved only for those representations which extend to a representation of  $\cp_+$ with non-trivial translations.


\providecommand{\bysame}{\leavevmode\hbox to3em{\hrulefill}\thinspace}


\begin{thebibliography}{10}

\bibitem{Araki1} H. Araki, {\it A lattice of von Neumann algebras
associated with the quantum field theory of a 	free Bose field},
J. Math. Phys. {\bf 4} (1963), 1343-1362.  


\bibitem{Araki3} H. Araki, {\it  Von Neumann algebras of local 
observables for free scalar field}, J. Math. Phys. {\bf 5} (1964), 1-13.

\bibitem{ArZs} H. Araki and L. Zsid\'o, {\it Extension of the structure theorem of Borchers and its application to half-sided modular inclusions},  Rev. Math. Phys.  {\bf 17}  (2005),  491-543.


\bibitem{BW1} J. Bisognano and E. Wichmann, {\it On the duality condition
for a Hermitian scalar field}, J. Math. Phys. {\bf 16} (1975), 985-1007
 
\bibitem{BW2} J. Bisognano and E. Wichmann, {\it On the duality
condition for quantum fields}, J. Math. Phys. {\bf 17} (1976),
303-321.

\bibitem{Borch68}
H.-J. Borchers,
{\it On the converse of the Reeh-Schlieder theorem},
Comm. Math. Phys. {\bf 10} (1968) 269Ð273. 


\bibitem{Borch92} H.-J. Borchers, {\it The CPT theorem in two-dimensional 
theories of local observables}, Commun. Math. Phys. {\bf 143} (1992), 315.

\bibitem{Bo1} H.-J.  Borchers,
{\it Half-sided modular inclusion and the construction of the Poincar\'e group},
Comm. Math. Phys. {\bf 179} (1996), 703-723. 


\bibitem{Bo2}  H.-J. Borchers,
{\it On Poincar\'e transformations and the modular group of the algebra associated with a wedge},
Lett. Math. Phys. {\bf 46} (1998), 295-301. 

\bibitem{BoBu} H.-J. Borchers and D. Buchholz, {\it Global properties of vacuum states in de Sitter space},  Ann. Inst. H. PoincarŽ Phys. ThŽor.  {\bf 70}  (1999),  23-40.

\bibitem{BoYn}  
 H.-J. Borchers and J. Yngvason, {\it On the PCT-theorem in the theory of local observables},  Mathematical physics in mathematics and physics (Siena, 2000),  39-64, Fields Inst. Commun. {\bf 30}, Amer. Math. Soc., Providence, RI, 2001.
 
 \bibitem{Bostelmann} H. Bostelmann, {\it Phase space properties and the short distance structure in quantum field theory}, J. Math. Phys. {\bf 46} (2005),  052301.
 
 \bibitem{bBratteli87}
O.~Bratteli and D.W. Robinson, {\em Operator Algebras and Quantum Statistical
  Mechanics $1$}, Springer Verlag, Berlin, 1987.

\bibitem{BrMo} J. Bros and U. Moschella, {\it Two-point functions and quantum fields in de Sitter 
universe}, Rev. Math. Phys. {\bf 8} (1996) 327Ð391. 

\bibitem{BrEpMo} J. Bros, H. Epstein and U. Moschella, 
{\it Analyticity properties and thermal effects for general quantum field theory on de Sitter space-time},
Comm. Math. Phys. {\bf 196} (1998),  535-570. 

\bibitem{BGL1} R. Brunetti, D. Guido and R. Longo, {\it Modular structure
and duality in conformal quantum field theory}, Commun.  Math.  Phys. 
{\bf 156} (1993), 201-219.

\bibitem{BGL2} R. Brunetti, D. Guido and R. Longo, {\it Group cohomology,
modular theory and spacetime symmetries}, Rev.  Math.  Phys.  {\bf 7}
(1994), 57-71.

\bibitem{BGL3} R. Brunetti, D. Guido and R. Longo, {\it Modular
localization and Wigner particles}, Rev.  Math.  Phys.  {\bf 14}
(2002), 759-786.

\bibitem{BuSu1} D. Buchholz and S.J.  Summers, {\it An algebraic characterization of vacuum states in Minkowski space}, Comm. Math. Phys.  {\bf 155}  (1993),   449-458.

\bibitem{BFS} D. Buchholz, M. Florig and S.J.  Summers, {\it An algebraic characterization of vacuum states in Minkowski space. II. Continuity aspects},
 Lett. Math. Phys.  {\bf 49}  (1999),   337-350.

\bibitem{BDFS} D. Buchholz, O. Dreyer, M. Florig and S.J.  Summers, {\it Geometric modular action and spacetime symmetry groups},
 Rev. Math. Phys.  {\bf 12}  (2000),   475-560.

\bibitem{BuSu2} D. Buchholz and S.J. Summers, {\it An algebraic characterization of vacuum states in Minkowski space. III. Reflection maps},
 Comm. Math. Phys.  {\bf 246}  (2004),  625-641.

\bibitem{BuSu3}  D. Buchholz and S.J. Summers, {\it Geometric modular action and spontaneous symmetry breaking}, Ann. Henri Poincar\'e  {\bf 6}  (2005),  607-624.
		
\bibitem{DR} S. Doplicher and J.E. Roberts, {\it Why there is a field algebra with a compact gauge group describing the superselection structure in particle physics}, Comm. Math. Phys.  {\it 131}  (1990), 51-107.

\bibitem{EO} J.P. Eckmann and K. Osterwalder, {\it An application of
Tomita's theory of modular Hilbert 	algebras: Duality for free Bose
fields}, J. Funct. Analysis {\bf 13} (1973), 1-22.

\bibitem{FaSc} L. Fassarella and B. Schroer, {\it Wigner particle theory and local quantum physics},
J. Phys. A {\bf 35} (2002),  9123-9164. 

\bibitem{FiGu2} F. Figliolini and D. Guido, {\it On the type of second quantization factors}, Journal of Operator Theory {\bf 31} (1994), 229-252.

\bibitem{Florig} M. Florig, {\it On Borchers' theorem},  Lett. Math. Phys.  {\bf 46}  (1998),  289--293.

\bibitem{FreHa} K. Fredenhagen and J. Hertel, 
{\it Local algebras of observables and pointlike localized fields},
Comm. Math. Phys. {\bf 80} (1981),  555-561. 

\bibitem{FreJo} K. Fredenhagen and M. J\"or\ss, 
{\it Conformal Haag-Kastler nets, pointlike localized fields and the existence of operator product expansions}, 
Comm. Math. Phys. {\bf 176} (1996),  541-554. 

\bibitem{GaFr} F. Gabbiani and J. Fr\"ohlich, 
{\it Operator algebras and conformal field theory}, 
Comm. Math. Phys. {\bf 155} (1993),  569-640. 


\bibitem{GH} G.W. Gibbons and S. Hawking, {\it Cosmological event 
horizon, thermodynamics, and particle creation}, Phys. Rev. D
{\bf 15} (1977), 2738-2752.

\bibitem{GuLo2} D. Guido and R. Longo, {\it An algebraic spin and
statistics theorem}, Commun. Math. Phys. {\bf 172} (1995), 517-533.

\bibitem{GuLo3} D. Guido and R. Longo, {\it The conformal spin and statistics theorem}, Commun. Math. Phys., {\bf 181} (1996), 11-35.

\bibitem{GuLo4} D. Guido and R. Longo, {\it Natural Energy Bounds in Quantum Thermodynamics}, 
Commun. Math. Phys. {\bf 218} (2001), 513-536.

\bibitem{HK} R. Haag, {\it Local Quantum Physics}, Springer-Verlag, New
York-Berlin-Heidelberg 1996.

\bibitem{HHK} R. Haag, N.M. Hugenholtz and M. Winnink, {\it On the equilibrium states in quantum statistical mechanics},  Comm. Math. Phys.  {\bf 5}  (1967), 215-236.


\bibitem{Haw} S.W. Hawking, {\it Particle creation by black holes},
Commun.  Math.  Phys.  {\bf 43}, 199 (1975).

\bibitem{HL} P. Hislop and R. Longo, {\it Modular structure of the local
observables associated with the free massless scalar field
theory}, Comm. Math. Phys. {\bf 84} (1982), 71-85.

\bibitem{KaeWie} R. K\"ahler and H-W. Wiesbrock, {\it Modular theory and the reconstruction of four-dimensional quantum field theories},  J. Math. Phys.  {\bf 42}  (2001),   74-86.

\bibitem{KuckSpin} B. Kuckert, {\it A new approach to spin and statistics},
Lett. Math. Phys. {\bf 35} (1995),  319-331. 

\bibitem{Kuckert} B. Kuckert, {\it BorchersÕ commutation relations and modular symmetries}, 
Lett. Math. Phys. {\bf 41} (1997), 307-320.

\bibitem{Lechner} G. Lechner 
{\it Construction of Quantum Field Theories with Factorizing S-Matrices},
Commun. Math. Phys. {\bf 277} (2008), 821Ð860.

\bibitem{Lip} R.L. Lipsman, {\it Group representations. A survey of some current topics}, Lecture Notes in Mathematics, Vol. 388. Springer-Verlag, Berlin-New York, 1974.

\bibitem{Lledo} F. Lled\'o, {\it Modular theory by example}, Contemporary Math.,   this volume.

\bibitem{Longobook}
R. Longo, {\it Real hilbert subspaces, modular theory, $SL(2, R)$ and CFT}, to appear in ``Von Neumann algebras in Sibiu'', pp. 33-91, Theta Foundation. 

\bibitem{Mund} J. Mund, {\it The Bisognano-Wichmann theorem for massive 
theories},  Ann. Henri Poincar\'e {\bf 2} (2001), 907-926. 

\bibitem{MSY} J. Mund, B. Schroer and J. Yngvason, {\it String-localized quantum fields and modular localization},  Comm. Math. Phys.  {\bf 268}  (2006),  621-672.

\bibitem{RS1} M. Reed and B. Simon, {\it Methods of modern mathematical physics. I. Functional analysis}, Academic Press, New York-London, 1972.

\bibitem{RS2} M. Reed and B. Simon, {\it Methods of modern mathematical physics. II. Fourier analysis, self-adjointness}, Academic Press, New York-London, 1975.

\bibitem{Rigotti} C. Rigotti, {\it On the essential duality condition for Hermitian scalar field},  Alg\`ebres d'op\'erateurs et leurs applications en physique math\'ematique (Proc. Colloq., Marseille, 1977),  pp. 307--320, Colloq. Internat. CNRS, 274, CNRS, Paris, 1979.

\bibitem{Sewell} G.L. Sewell, {\it Relativity of temperature and the Hawking effect},
Phys. Lett. A {\bf 79} (1980),  23-24. 

\bibitem{Simms} D.J. Simms, {\it Lie groups and quantum mechanics},
Lecture Notes in Mathematics, No. 52 Springer-Verlag, Berlin-New York 1968.

\bibitem{SW} R.F. Streater and A.S. Wightman, {\it PCT, spin and
statistics, and all that}, Addison Wesley, Reading, MA 1989.

\bibitem{ToTa} M. Takesaki, {\it Tomita's theory of modular Hilbert algebras and its applications},
Lecture Notes in Mathematics, Vol. 128 Springer-Verlag, Berlin-New York 1970.

\bibitem{bTakesakiI}
M.~Takesaki, {\em Theory of Operator Algebras I}, Springer Verlag, Berlin,
  2002.

\bibitem{bTakesakiII}
\bysame, {\em Theory of Operator Algebras II}, Springer Verlag, Berlin, 2003.

\bibitem{bTakesakiIII}
\bysame, {\em Theory of Operator Algebras III}, Springer Verlag, Berlin, 2003.


\bibitem{U} W.G. Unruh, {\it Notes on black hole
evaporation}, Phys. Rev. {\bf D14} (1976), 870.

\bibitem{Wies1} H-W. Wiesbrock, {\it Conformal quantum field theory and half-sided modular inclusions of von Neumann algebras},  Comm. Math. Phys.  {\bf 158}  (1993),  537-543.

\bibitem{Wies2} H.W. Wiesbrock, {\it Modular intersections of von Neumann algebras in quantum field theory},  Comm. Math. Phys.  {\bf 193}  (1998),   269-285.

\bibitem{Yng1} J. Yngvason, {\it Zero-mass infinite spin representations of the Poincar\'e group and quantum field theory}, Comm. Math. Phys.  {\bf 18}  (1970), 195-203.

\end{thebibliography}
\end{document}